\documentclass[12pt,reqno]{amsart}%,reqno
\usepackage{amssymb,amsmath,amsthm}
\usepackage{a4}
\usepackage{enumitem}
\usepackage{color}
\usepackage[colorlinks=true,
linkcolor=webgreen,
filecolor=webbrown,
citecolor=webgreen]{hyperref}
\definecolor{webgreen}{rgb}{0,0,1}%{1,0.0,.6}
\definecolor{recrown}{rgb}{1,.2,.6}
\usepackage{fullpage}
\usepackage{orcidlink}

\begin{document}
\newtheorem{theorem}{Theorem}
\newtheorem{corollary}[theorem]{Corollary}
\newtheorem{lemma}[theorem]{Lemma}
\theoremstyle{definition}
\newtheorem{example}{Example}
\newtheorem*{examples}{Examples}
\newtheorem*{notation}{Notation}
\theoremstyle{theorem}
\newtheorem{thmx}{Theorem}
\renewcommand{\thethmx}{\text{\Alph{thmx}}}% "letter-numbered" theorems
\newtheorem{lemmax}{Lemma}
\renewcommand{\thelemmax}{\text{\Alph{lemmax}}}% "
%\leftmargin=.5in
%\rightmargin=0.5in
%\textwidth=6truein
%\textheight=11.6truein
\hoffset=-0cm
%\voffset=+-2cm
\theoremstyle{definition}
\newtheorem*{definition}{Definition}
\newtheorem*{remark}{\bf Remark}
\title{\bf Prime numbers and factorization of polynomials}
\author{Jitender Singh$^{1,\dagger}$ {\large \orcidlink{0000-0003-3706-8239}}}
\address[1]{Department of Mathematics,
Guru Nanak Dev University, Amritsar-143005, India\newline %\linebreak
 {\tt jitender.math@gndu.ac.in}}
\markright{}
\date{}
\footnotetext[2]{Corresponding author email(s): %{\tt jitender.math@gndu.ac.in}

%\url{https://sites.google.com/view/sonumaths3/home}\\

2020MSC: {Primary 11R09; 12J25; 12E05; 11C08}\\

}
\maketitle
\newcommand{\K}{\mathbb{K}}
\begin{abstract}
In this article, we obtain upper bounds on the number of irreducible factors of some classes of  polynomials having integer coefficients, which in particular yield some of the well known irreducibility criteria.  For devising our results, we use the information about prime factorization of the values taken by such polynomials at sufficiently large integer arguments along with the information about their root location in the complex plane. Further, these techniques are extended to bivariate polynomials over arbitrary fields using non-Archimedean absolute values, yielding extensions of the irreducibility results of M. Ram Murty and S. Weintraub to bivariate polynomials.
\end{abstract}
\section{Introduction}
Irreducibility of polynomials having integer coefficients has been an active area of research since the classical irreducibility criteria due to Sch\"onemann \cite{S}, Eisenstein \cite{E}, Perron \cite{P}, and Dumas \cite{Du}. Irreducible polynomials are fundamental in mathematics and application in computer science. A number of new criteria for testing irreducibility of  polynomials having integer coefficients have been emerged in the recent past and for a quick review, the reader may refer to see \cite{SKJS2023}. Irreducible polynomials can produce prime numbers for finitely many consecutive integer arguments. In view of this, Euler \cite{Euler} in 1772 showed that the polynomial $41-x+x^2$ produces  prime numbers for 40 consecutive numbers for $x$ from $1$ to $40$ (see for detail, the sequence A005846 in \cite{NJASloane}), where we note that the polynomial $41-x+x^2$ is irreducible in $\mathbb{Z}[x]$.  For many other such examples of prime producing polynomials, the reader may refer to see the sequences A007641, A050267, A050268 and cross references therein in The On-Line Encyclopedia of Integer Sequences \cite{NJASloane}.
%, since 41 is prime and the zeros of this polynomial lie outside the closed unit disc in the complex plane.
On the other hand, prime numbers can produce irreducible polynomials. This is evident from the irreducibility criterion of A. Cohn (see P\'olya and Szeg\"o \cite[p.~113]{Polya}) which states that if  a prime number $p$ has decimal expansion
\begin{eqnarray*}
p=a_0+a_1 10^1+\cdots+a_n 10^n,~0\leq a_i\leq 9,
\end{eqnarray*}
then the polynomial
\begin{eqnarray*}
a_0+a_1 x+\cdots+a_n x^n
\end{eqnarray*}
is irreducible in $\mathbb{Z}[x]$. This result was extended to arbitrary base $\geq 2$ in \cite{Brillhart,Bonciocat4}. Another instance witnessing the affinity between prime numbers and irreducible polynomials is the Buniakowski's conjecture (1854) which states that if $f\in \mathbb Z[x]$ is an irreducible polynomial such that the integers in the set $f (\mathbb N)$ have no factor in common other than $\pm 1$, then $f$ takes prime values infinitely often. Buniakowski's conjecture is still open, but its converse is true, that is, if $f$ takes prime values for infinitely many values of integer arguments, then the polynomial $f$ must be irreducible in $\mathbb{Z}[x]$. In view of this, Murty in \cite{Mu} proved the following  exquisite irreducibility criterion.
\begin{thmx}[\cite{Mu}]\label{M} Let $f=a_0+a_1x+\cdots +a_nx^n\in \mathbb Z[x]$ be a polynomial of degree $n$ and define
\begin{eqnarray*}
H_f=\max_{0\leq i\leq n-1}|a_i|/|a_n|.
\end{eqnarray*}
If $f(m)$ is a prime number for some integer $m\geq H_f+2$, then the polynomial $f$ is irreducible in $\mathbb{Z}[x]$.
\end{thmx}
In \cite{G}, Girstmair extended Theorem \ref{M} to the case when $f$ takes a composite value equal to a prime number times an integer at a sufficiently large integer argument and obtained the following result.
\begin{thmx}[\cite{G}]\label{G}
Let $f=a_0+a_1x+\cdots +a_nx^n\in \mathbb Z[x]$ be a primitive polynomial of degree $n$. Suppose there exist natural numbers $d$,  $m$, and a prime $p\nmid d$  with $m\geq H_f+d+1$ and $f(m)=\pm d\cdot p$. Then the polynomial $f$ is irreducible in $\mathbb{Z}[x]$.
\end{thmx}
Let $f^{(0)}(x)=f(x)$, and for each natural number $i$, $f^{(i)}(x)$ denote the $i$th derivative of $f$ with respect to $x$.
Recently, in \cite{J-S-2}, the authors obtained the following generalization of Theorem \ref{G} by requiring $f$ to take a prime power value multiplied with an integer at a sufficiently large integer argument.
\begin{thmx}[\cite{J-S-2}]\label{J-S-2b}
Let $f=a_0+a_1 x+\cdots+a_nx^n\in \Bbb{Z}[x]$ be a primitive polynomial  of degree $n$. Suppose there exist natural numbers $m$, $d$, $k$, $j\leq n$, and a prime $p\nmid d$ such that $m\geq H_f+d+1$, $f(m)=\pm p^k d$, $\gcd(k,j)=1$, $p^k$ divides ${f^{(i)}(m)}/{i!}$ for each index $i=0,\ldots,j-1$, and for $k>1$, also $p$ does not divide ${f^{(j)}(m)}/{j!}$. Then the polynomial $f$ is irreducible in $\Bbb{Z}[x]$.
\end{thmx}
Note that Theorem \ref{G} is precisely the case $k=1$ of Theorem \ref{J-S-2b}.

For a positive divisor $d$ of a natural number $a$, if $d$ and $a/d$ are coprime, we say that $d$ is a unitary divisor of $a$. We shall write $d\| a$ to mean that $d$ is a unitary divisor of $a$. Now a closer look at the aforementioned three  irreducibility criteria reveals that in each case for a sufficiently large integer $m$, the number $|f(m)|$ is assumed to be a product of a pair of its unitary divisors. In fact this simple common observation about Theorems \ref{M}-\ref{J-S-2b} leads to the following result.
\begin{theorem}\label{th1}
Let $f=a_0+a_1x+\cdots +a_nx^n\in \mathbb Z[x]$ be a polynomial of degree $n$. Suppose there exist positive integers $m$ and $d$ such that $m\geq H_f+d+1$, $d\|f(m)$, and $|f(m)/d|>1$. Let $\nu_m$ denote the number of prime factors of $|f(m)/d|$ counted with multiplicities. Then the polynomial $f$ is a product of at most $\nu_m$ irreducible factors in $\mathbb Z[x]$. In particular, if $\nu_m=1$ for such an $m$, then $f$ is irreducible in $\mathbb Z[x]$.
\end{theorem}
Of course only those $m$'s would be relevant for a given polynomial in Theorem \ref{th1} for which $\nu_m$ does not exceed the degree of the underlying polynomial.
\begin{remark}
We mention here that the upper bound on the number of irreducible factors of $f$ in Theorem \ref{th1} is best possible in the sense that there exist examples of polynomials for which these bounds are attained. The simplest example of such a nonconstant polynomial is the polynomial $x^n$ which takes the value $p^n$ having $n$ factors at any prime number $p\geq 3$. To see more interesting examples, let us consider the polynomial
\begin{eqnarray*}
f_n(x)&=& x^{2n}-x^2-2x-1,~n\geq 3
\end{eqnarray*}
for which $H_f=2$. For $n=2$, we choose $d=1$ and $m=3\geq (2+d)$, and we find that $f_2(3)= 65=5\times 13$, a product of two distinct primes. This in view of Theorem \ref{th1} tells us that $\nu_{3}=2$. On the other hand, we have
%\begin{eqnarray*}
$f_2(x)= x^{4}-(x+1)^2=(x^2-x-1)(x^2+x+1)$,
%\end{eqnarray*}
a product of two irreducible polynomials. So, $f_2$ is a product of exactly $2=\nu_{3}$ irreducible factors in $\mathbb{Z}[x]$. A similar observation follows for few other values of $n$ as given in Table \ref{table1}.
\begin{table}[h!!!]
    \caption{Factorization of $f_n$ for some values of $n$}\label{table1}
    \begin{tabular}{llllcc}
      \hline
      $n$ & $m$ & $d$ & $f_n(m)/d$ & $\nu_m$ & No. of Irreducible\\ &&&(Prime factorization)&&factors of $f_n(x)$ \\
      \hline
      3 & 3 & 1 & $23\times 31$ & 2  & $2$\\
      4 & 5 & 1 & $619\times 631$ & 2  & $2$\\
      5 & 3 & 1 & $13\times 19\times 239$ & $3$  & $3$\\
      6 & 6 & 1 & $46649\times 46663$ & $2$  &$2$\\
      7 & 21 & 1 & $1801088519\times 1801088563$ & 2  &$2$\\
      8 & 50 & 11 & $2551\times 1392056591\times 39062499999949$ & 3  & $3$\\
      9 & 9 & 1 & $387420479\times 387420499$ & 2  & $2$\\
      10 & 7 & 3 & $94158419\times 282475241$ & 2  & $2$\\
      \hline
    \end{tabular}
    \end{table}

Recall that the reciprocal polynomial $\tilde{f}\in \mathbb{Z}[x]$ of $f\in \mathbb{Z}[x]$ is defined as
%\begin{eqnarray*}
$\tilde{f}(x)=x^{\deg f}f(1/x)$.
%\end{eqnarray*}
Since each of $f$ and $\tilde{f}$ have same number of irreducible factors, it is more effective to simultaneously compute factorization of $f(m)$ and $\tilde{f}(m)$ and choose $m$ at which one of  $f$ and $\tilde{f}$ have smaller number of prime factors. For example, the smallest value of $m$ for which $f_8(m)$ has 3 prime factors is 50 with $d=11$, whereas a much smaller value $m=3$ is required for the corresponding reciprocal polynomial $\tilde{f_8}$ to attain three prime factors, since
\begin{eqnarray*}
\tilde{f_8}(3)=76527503=13\times 673\times8747.
\end{eqnarray*}
\end{remark}
\begin{remark}
For a positive integer $a$, we consider the polynomial
\begin{eqnarray*}
f_a=x^4+(4a+2)x^2+1.
\end{eqnarray*}
If for a combination of $a$ and $m$ with $m\geq H_{f_a}+2$, we have $\nu_m=2$ for $f_a$, then by Theorem  \ref{th1}, the polynomial $f_a$ is a product of at most two irreducible polynomials in $\mathbb Z[x]$. We show that $f_a$ is irreducible. If possible, let $f_a(x)=g(x)h(x)$ for nonconstant polynomials $g$ and $h$ in $\mathbb Z[x]$. Clearly, $f_a$ has no rational root, and so, $f$ has no linear factor. It follows that each of $g$ and $h$ is a monic polynomial of degree $2$. So, we must have $g=\pm1+a_1 x+x^2$ and $h=\pm 1+b_1x+x^2$ for some $a_1,b_1\in \mathbb Z$. We then have
\begin{eqnarray*}
x^4+(4a+2)x^2+1 &=&  (\pm1+a_1 x+x^2)(\pm 1+b_1x+x^2)\\&=&1+(a_1+b_1)x+(2+a_1b_1)x^2+(a_1+b_1)x^3+x^4,
\end{eqnarray*}
which tells us that $a_1=-b_1$ and $2+a_1b_1=4a+2$ or $-a_1^2=4a$, which is impossible. Thus, such an $f_a$ must be irreducible in $\mathbb Z[x]$.

In particular, for $a=2$, we take $m=12=H_f+2$ so that  $f_2(12)=22177=67\times331$, and so $\nu_m=2$ for $f_2$. By preceding paragraph we infer that the polynomial $f_2$ is irreducible in $\mathbb Z[x]$. On the other hand, to deduce the irreducibility of $f_2$ using Theorem \ref{G}, one requires to take at least $m=15$, where  $f_2(15)=4\times 13219$, and Theorem \ref{M} applies to $f_2$ for at least $m=18$ with the prime value $f(18)=108217$.  For more appealing example of this kind corresponds to the case $a=5$, where Theorem \ref{th1} applies for $m=H_{f_5}+2=24$, where $f_5(24)=7\times 49207$, whereas Theorem \ref{G} applies for at least $m=30$ and Theorem \ref{M} applies for $m=204$.
\end{remark}
Our second factorization result generalizes Theorem \ref{J-S-2b} and is stated as follows.
\begin{theorem}\label{th2}
Let $f=a_0+a_1x+\cdots +a_nx^n\in \mathbb Z[x]$ be a polynomial with $a_0a_n\neq 0$. Suppose there exist  natural numbers $m$, $d$, $k$, $j\leq n$, and a prime $p\nmid d$ such that $m\geq H_f+d+1$, and $|f(m)|/d=p^k$ is coprime to $f^{(j)}(m)/j!$. Then the polynomial $f$ is a product of at most $\min\{k,j\}$  irreducible factors in $\mathbb Z[x]$.  In particular,  if $k=1$, or $j=1$ for such an $m$, then $f$ is irreducible.
\end{theorem}
Here too, the upper bound as mentioned in Theorem \ref{th2} is best possible. To see this, first we observe inductively that for any positive integer $n$, the maximum value of the binomial coefficient $n\choose r$ for $r=0,\ldots,n$ is one of ${n\choose \lceil n/2\rceil}$ and ${n\choose \lfloor n/2\rfloor}$. Further, using induction on $n$ again, we arrive at the following two inequalities.
\begin{eqnarray*}
{n\choose \lceil n/2\rceil}\leq 2^{n-1};~{n\choose \lfloor n/2\rfloor}\leq 2^{n-1}.
\end{eqnarray*}
Using these facts, we observe that the polynomial
\begin{eqnarray*}
f &=&  (x-1)^n,~n\geq 2,
\end{eqnarray*}
satisfies the hypothesis of Theorem \ref{th2} for $d=1$, $k=n^2$, $j=n$,
$m=2^n+1>2+2^{n-1}\geq H_f+d+1$, since here,
either $H_f={n\choose \lceil n/2\rceil}$  or $H_f={n\choose \lfloor n/2\rfloor}$ and $|f(m)|/d=2^{n^2}$ which is coprime to $1=f^{(n)}(m)/n!$. So, $f$ is a product of at most $\min\{n^2,n\}=n$ irreducible factors in $\mathbb{Z}[x]$. We note that the polynomial $(x-1)^n$ is a product of exactly $n$ irreducible factors in $\mathbb{Z}[x]$.

As an extension of the irreducibility criterion of A. Cohn, we have the following factorization result whose ideas of the proof stem from the paper \cite{Mu}.
\begin{theorem}\label{th3}
    Let $b\geq 2$, $N\geq 2$ be positive integers. Let $N$ has the  $b$-adic expansion
    \begin{eqnarray*}
      N &=& a_0+a_1 b+a_2 b^2+\cdots+a_n b^n.
    \end{eqnarray*}
Let $\nu_N$ denote the number of prime divisors of $N$ counted with multiplicities. Then the polynomial $f=a_0+a_1 x+\cdots+a_n x^n$ is a product of at most $\nu_N$ irreducible factors in $\mathbb{Z}[x]$. In particular, if $\nu_N=1$, then $f$ is irreducible.
\end{theorem}
As before, the upper bound mentioned in Theorem \ref{th3} is best possible as follows from the following explicit example.
For a prime number  $q$ with
\begin{eqnarray*}
q>1+ \max\Bigl\{{n\choose \lceil n/2\rceil},{n\choose \lfloor n/2\rfloor}\Bigr\},
\end{eqnarray*}
the number $q^n$ has the $(q-1)$-adic expansion
\begin{eqnarray*}
q^n &=& {n\choose 0}+{n\choose 1}(q-1)+\cdots+{n\choose n}(q-1)^n.
\end{eqnarray*}
Since $q^n$ has $n$ prime factors counted with multiplicities, by Theorem \ref{th3}, the polynomial
${n\choose 0}+{n\choose 1}x+\cdots+{n\choose n}x^n$
is a product of at most $\nu_{q^n}=n$ irreducible factors in $\mathbb{Z}[x]$. Using binomial theorem, we see that ${n\choose 0}+{n\choose 1}x+\cdots+{n\choose n}x^n =(x+1)^n$, which is  a product of exactly $n$ irreducible factors in $\mathbb{Z}[x]$.

We may have the following extension of Theorem \ref{M} for bivariate polynomials.
\begin{theorem}\label{th4}
 Let $\K$ be a field, and let $f=a_0(x)+a_1(x)y+\cdots+a_n(x)y^n\in \K[x,y]$ be a  polynomial with $a_0a_n\neq 0$, and for a real number $\rho>1$, define
 \begin{eqnarray*}
 H_f=\rho^{\max_{0\leq i\leq n-1}\deg a_i-\deg a_n}.
 \end{eqnarray*}
  Suppose there exists a polynomial $a\in \K[x]$ for which
 \begin{eqnarray*}
              \deg a\geq  \frac{\log(H_f+2)}{\log \rho},
 \end{eqnarray*}
  and $\nu_a$ is the number of irreducible factors of the polynomial $f(x,a(x))$ in $\K[x]$ counted with multiplicities. Then  the polynomial $f(x,y)$ is a product of at most $\nu_a$ irreducible factors in $\K[x,y]$. In particular, if $\nu_a=1$, then $f$ is irreducible in $\K[x,y]$.
\end{theorem}
The upper bound in Theorem \ref{th4} on irreducible factors of $f(x,y)$ is attained by the polynomial $f(x,y)=(x+y)^n\in \mathbb{Q}[x,y]$, since, here $H_f=\rho^n$, and if we take $\rho=2$ and $a(x)=2+x^{n+1}-x$, we observe that $\deg a(x)=n+1\geq \log(2+2^n)/\log(2)$, and $f(x,a(x))=(2+x^{n+1})^n$, where we note that the polynomial $2+x^{n+1}$ is Eisensteinian with respect to the prime $2$. It follows that $\nu_a=n$. The polynomial $x+y$ is irreducible in $\mathbb{Q}[x,y]$, since $x+y$ is Eisensteinian with respect to the irreducible polynomial $x$, so the polynomial $(x+y)^n$ is a product of exactly $n$ irreducible factors in $\mathbb{Q}[x,y]$.

 The authors in \cite{JRG2024}  proved the following generalization of the well known Weintraub's irreducibility result \cite{W} for univariate polynomials.
\begin{thmx}[\cite{JRG2024}]\label{thD}
        Let $f=a_0+ a_{1}x+\cdots+a_n x^n\in \Bbb{Z}[x]$ be a polynomial and suppose there exists a prime  $p$ and   positive integers $k$ and $j$ with $1\leq j\leq n$ such that  $p$ does not divide $a_j$, $p^k$ divides $a_i$ for $i=0,\ldots,j-1$. If there exists a least integer $\ell$ with $0\leq \ell\leq j-1$ for which $p^{k+1}$ does not divide $a_\ell$ and $\gcd(j-\ell,k)=1$, then any factorization $f(x)=g(x)h(x)$ of $f$ in $\mathbb{Z}[x]$ has a factor of degree at most $n-j+\ell$.
\end{thmx}
Weintraub's irreducibility criterion is precisely the case $j=n$ of Theorem \ref{thD}.

We may have the following bivariate version of Theorem \ref{thD}, whose idea of the proof is based on Dumas Theorem \cite{Du} on Newton-polygons.
\begin{theorem}\label{th5} Let $\K$ be a field, and let
         $f=a_0(x)+ a_{1}(x)y+\cdots+a_n(x) y^n\in \K[x,y]$ be a polynomial with $a_0a_n\neq 0$. Suppose there exist an integer $j$ and a least integer $\ell$ satisfying $1\leq \ell+1\leq j\leq n$ such that
         \begin{eqnarray*}
         a_j\in \K^\times,~\frac{\deg a_\ell}{j-\ell}=\max_{0\leq i\leq j-1}\Bigl\{\frac{\deg a_i}{j-i}\Bigr\},~\gcd(j-\ell,\deg a_\ell)=1.
         \end{eqnarray*}
    Then any factorization $f(x,y)=g(x,y)h(x,y)$ of $f$ in $\K[x,y]$ has a factor of degree at most $n-j+\ell$.
\end{theorem}
Rest of the paper is organized as follows. The proofs of our main results (Theorems \ref{th1}-\ref{th5}) are given in Section \ref{sec:2}. In Section \ref{sec:3}, some examples of polynomials are discussed whose factorization properties may be deduced using our result.
\section{Proofs of Theorems \ref{th1}-\ref{th5}.}\label{sec:2}
To prove Theorem \ref{th1}, we recall that for any polynomial $f=a_0+a_1x+\cdots +a_nx^n\in \mathbb Z[x]$, and
any $x\in \mathbb{C}$ satisfying $|x|\geq H_f+1$, we have $-1/|x|>-1/(H_f+1)$; $-|a_i|/|a_n|\geq (-H_f)$. Consequently, we have \begin{eqnarray*}
\frac{|f(x)|}{|a_n||x|^n} &\geq &1-\sum_{i=0}^{n-1}\frac{|a_i|}{|a_n||x|^{n-i}}\geq  1-\sum_{i=0}^{n-1}\frac{H_f}{(H_f+1)^{n-i}}=\frac{1}{(H_f+1)^{n}}>0,
\end{eqnarray*}
which shows that each zero $\theta$ of $f$ satisfies $|\theta|<H_f+1$.
 \begin{proof}[\bf Proof of Theorem \ref{th1}] Suppose that $f(x)=f_1(x)f_2(x)\cdots f_r(x)$ is a product of $r$  irreducible polynomials $f_i\in\mathbb Z[x]$, $1\leq i\leq r$. We may assume without loss of generality that each of $f_i$'s is nonconstant and that $r\geq 2$. We  have $1\leq |f(m)|=|f_1(m)||f_2(m)|\cdots |f_r(m)|$, and so, $|f_i(m)|\geq 1$ for each $i$. Further, for each $i=1,\ldots, r$, if $\alpha_i$ is the leading coefficient of $f_i$, then we may express $f_i(x)$ as $f_i(x)=\alpha_i\prod_{\theta}(x-\theta)$, where the product runs over all the zeros of $f_i$. Combining this with the hypothesis that $m\geq H_f+d+1$ and the fact that each  zero $\theta$ of $f$ (and hence each zero of $f_i$) satisfies $|\theta|<H_f +1$ or $-|\theta|>-(H_f+1)$, we arrive at the following:
 \begin{eqnarray*}
 |f_i(m)| = |\alpha_i|\prod_{\theta}|m-\theta|&\geq&  |\alpha_i|\prod_{\theta}(m-|\theta|)\\&>&|\alpha_i|\prod_{\theta}(H_f+d+1-(H_f+1))\\&=&|\alpha_i|d^{\deg(f_i)}\geq d.
 \end{eqnarray*}
Consequently, we have $|f_i(m)|>d\geq 1$ for all $i=1,\ldots,r$. This in view of the fact that $d\| f(m)$ and that $|f(m)/d|>1$, shows that there exists a prime factor  of $|f_i(m)|$ coprime to $d$ for each $i=1,\ldots, r$.  Thus, we have $r\leq \nu_m$.
\end{proof}
%\section{Proof of Theorem \ref{th2}.}
\begin{proof}[\bf Proof of Theorem \ref{th2}]  As before, we assume that $f(x)=f_1(x)f_2(x)\cdots f_r(x)$ is a product of $r$ irreducible polynomials $f_i\in\mathbb Z[x]$, $1\leq i\leq r$. We may assume without loss of generality that each $f_i$ is nonconstant and $r\geq 2$. Since $|f(m)/d|=p^k>1$ and $p\nmid d$, it follows that $d\|f(m)$. By Theorem \ref{th1}, we have $r\leq \nu_m=k$. Further, proceeding as in the proof of Theorem  \ref{th1}, we find that $|f_i(m)|>d$ for all $i=1,\ldots,r$. Consequently, $p$ divides $|f_i(m)|$ for all $i=1,\ldots,r$.

Suppose on the contrary that $r>j$. Then we have
\begin{eqnarray*}
\frac{f^{(j)}(m)}{j!} &=& \sum_{i_1+\ldots+i_r=j} \frac{f_1^{(i_1)}(m)}{i_1!}\cdots \frac{f_r^{(i_r)}(m)}{i_r!},
\end{eqnarray*}
where each of the indices under the summation satisfies $0\leq i_s\leq r$, $s=1,\ldots,r$. This in view of the fact that $r>j$ tells us that for each such $r$-tuple $i_1,\ldots,i_r$, there exists at least one index $s\in \{1,\ldots,r\}$  for which $i_s=0$. Corresponding to this particular index $s$, we have $f_s(m)=\frac{f_s^{(i_s)}(m)}{i_s!}$  is a factor of the product $\frac{f_1^{(i_1)}(m)}{i_1!}\cdots \frac{f_r^{(i_r)}(m)}{i_r!}$, which in view of the fact that $p$ divides $|f_s(m)|$  shows that $p$ divides the number $\frac{f_1^{(i_1)}(m)}{i_1!}\cdots \frac{f_r^{(i_r)}(m)}{i_r!}$. Consequently, $p$ divides the number $\sum_{i_1+\ldots+i_r=j} \frac{f_1^{(i_1)}(m)}{i_1!}\cdots \frac{f_r^{(i_r)}(m)}{i_r!}=\frac{f^{(j)}(m)}{j!}$, which contradicts the hypothesis, and so, we must have $r\leq j$.
\end{proof}
%\section{Proof of Theorem \ref{th3}.}
To prove Theorem \ref{th3}, we will make use of the following two lemmas of Murty \cite{Mu}.
\begin{lemmax}[\cite{Mu}]\label{lemma1}
    Let $f=a_0+a_1 x+\cdots+a_n x^n\in \mathbb{Z}[x]$ with $n\geq 2$, $a_n\geq 1$, $a_{n-1}\geq 0$, $a_i\leq H_f$ for some positive real number $H_f$ for each $i=0,\ldots,n-2$. Then any complex zero of $f$ either has nonpositive real part or
    \begin{eqnarray*}
|\alpha|<\frac{1+\sqrt{1+4H_f}}{2}.
    \end{eqnarray*}
\end{lemmax}
 If we take $\alpha=u+\iota v$, $u,v\in \mathbb{R}$ in Lemma \ref{lemma1}, then the condition $u\leq 0$ implies that $|b-\alpha|\geq |\mathfrak{R}(b-\alpha)|=b-u\geq b>1$. On the other hand if $u>0$ and $h_f\geq 2$, then we observe that $\frac{1+\sqrt{1+4H_f}}{2}\leq H_f$. Consequently,  if we take $b\geq (H_f+1)\geq 3$,  then we have $|b-\alpha|\geq b-|\alpha|> b-\frac{1+\sqrt{1+4H_f}}{2}=\geq H_f+1-H_f=1$. So, for $b\geq 3$, we have $|b-\alpha|>1$.
\begin{lemmax}[\cite{Mu}]\label{lemma2}
   Suppose that $\alpha$ is a zero of the polynomial  $f=a_0+a_1 x+\cdots+a_n x^n$ with coefficients $a_i\in \{0,1\}$, $a_n=1$. If $\text{arg}(\alpha)\leq \pi/4$, then $|\alpha|<3/2$. Otherwise real part of $\alpha$ is less than $\frac{1+\sqrt{5}}{2\sqrt{2}}<3/2$.
\end{lemmax}
\begin{proof}[\bf Proof of Theorem \ref{th3}] Let $f(x)=f_1(x)f_2(x)\cdots f_r(x)$ be a product of $r$ irreducible polynomials in $f_i\in\mathbb Z[x]$, $1\leq i\leq r$. We may assume without loss of generality that each of $f_i$'s is nonconstant and $r\geq 2$. Then $N=f_1(b)\cdots f_r(b)$ which shows that $|f_i(b)|\geq 1$ for each $i$. Let $f_i(x)=c_i \prod_{\theta} (x-\alpha)$ where $c_i$ is the leading coefficient of $f_i$ and $\alpha$ runs over all zeros of $f_i$.

First we assume that $b>2$. Then in view of the discussion following Lemma \ref{lemma1}, we have that  each zeros $\alpha$ of $f_i$ satisfies $|b-\alpha|>1$. Consequently,  we have
\begin{eqnarray*}
|f_i(b)|=|c_i|\prod_\alpha |b-\alpha|>|c_i|\geq 1,
\end{eqnarray*}
which shows that $f_i(b)$ is a nontrivial factor of $f(b)=N$, and so, $r\leq \nu_N$.

Now assume that $b=2$ so that we may take $H_f=1$. Let $v=3/2$. By Lemma \ref{lemma2}, each zero $\alpha$ of $f$ satisfies $\mathfrak{R}(\alpha)<v$. Fix $i\in \{1,\ldots,r\}$, and define
\begin{eqnarray*}
g_i(x)=f_i(x+v) &=& c_i\prod_\alpha (x+v-\alpha),
\end{eqnarray*}
wherein $\alpha$ runs over all zeros of $f_i$. If a zero $\alpha$ of $f_i$ is real then the polynomial $x+v-\alpha$ has all nonnegative coefficients. On the other hand if a zero $\alpha$ of $f_i$ is complex, then the complex conjugate $\bar{\alpha}$ is also a zero of $f_i$, and we have
\begin{eqnarray*}
(x+v-\alpha)(x+v-\bar{\alpha})=x^2+2\mathfrak{R}(v-\alpha)x+|v-\alpha|^2,
\end{eqnarray*}
which has again nonnegative coefficients. These observations show that $g_i$ is a polynomial having either all nonnegative coefficients or all nonpositive coefficients depending upon the sign of $c_i$. So, $|g_i(-x)|\leq |g_i(x)|$ for all $x>0$. Since $v=3/2<2$, we have
\begin{eqnarray*}
|f_i(1)|=|g_i(-1/2)|=|g_i(-2+v)|<|g_i(2-v)|=|g_i(1/2)|=|f_i(2)|.
\end{eqnarray*}
Since $f_i$ is irreducible in $\mathbb{Z}[x]$, we must have $f_i(1)\neq 0$. So, we have $1\leq |f_i(1)|<|f_i(2)|$. Thus, we have $|f_i(2)|>1$ for each $i=1,\ldots,r$, and so, $r\leq \nu_N$.
\end{proof}
To prove Theorem \ref{th4}, we define an absolute value $\|a\|$ of any nonzero element $a\in \K[x]$ to be $\rho^{\deg a}$ where $\deg 0=-\infty$ and $\rho^{-\infty}=0$. For $a,b\in \K[x]$, define $\|a/b\|=\|a\|/\|b\|$, which extends $\|\cdot\|$ from $\K[x]$ to $\K(x)$. We fix an algebraic closure $\overline{\K(x)}$ of $\K(x)$, and we fix an extension of the absolute value $\|\cdot\|$ from $\K(x)$ to $\overline{\K(x)}$ (see \cite{EP}) which we shall also denote by $\|\cdot\|$. The absolute value $\|\cdot\|$ is non-Archimedean, that is, for all $a,b\in \overline{\K(x)}$, we have $\|a+b\|\leq \max\{\|a\|,\|b\|\}$. Observe that $\|\cdot\|$ satisfies the triangle inequality and $\|\cdot\|\geq 1$ on $\K[x]$. Further, $\|ab\|=\|a\|\|b\|$ for all $a,b\in \overline{\K(x)}$. We shall make use of these properties of $\|\cdot\|$ to prove Theorem \ref{th4} as follows:
\begin{proof}[\bf Proof of Theorem \ref{th4}] First note that $\deg a\geq  \log (H_f+2)/\log \rho$ if and only if $\|a\|\geq H_f+2$.  For any $y\in \overline{\K(x)}$ satisfying $\|y\|\geq H_f+1$, we have on using properties of $\|\cdot\|$ that
\begin{eqnarray*}
\frac{\|f(x,y)\|}{\|y^n a_n\|} \geq 1-\sum_{i=1}^{n-1}\frac{\|a_i\|}{\|a_n\|}\frac{1}{\|y^{n-i}\|}
 &= & 1-\sum_{i=1}^{n-1}\rho^{\deg a_i-\deg a_n}\frac{1}{\|y\|^{n-i}}\\
 &\geq & 1-\rho^{\max_{i}\deg a_i-\deg a_n}\sum_{i=1}^{n-1}\frac{1}{H_f^{n-i}}\\
   &\geq & 1-H_f\sum_{i=1}^{n-1}\frac{1}{H_f^{n-i}}= \bigl(H_f+1\bigr)^{-n}>0.
\end{eqnarray*}
This shows that any zero $\theta\in \overline{\K(x)}$ of $f(x,y)=a_0+a_1 y+\cdots+a_ny^n\in \K[x,y]$ satisfies $\|\theta\|<H_f+1$.

Now assume that $f(x,y)=f_1(x,y)\cdots f_r(x,y)$ be a product of $r$ irreducible polynomials $f_i$'s in $\K[x,y]$. Then $\|f_i(x,a(x))\|\geq 1$ for every $i=1,\ldots,r$. If $\alpha_i(x)\in \K[x]$ is the leading coefficient of $f_i$, then we may write
\begin{eqnarray*}
f_i(x,y) &=& \alpha_i(x)\prod_{\theta} (y-\theta),
\end{eqnarray*}
where the product runs over all zeros of $f_i$ in $\overline{\K(x)}$. We then have
\begin{eqnarray*}
\|f_i(x,a(x))\| = \|\alpha_i(x)\|\prod_{\theta} \|a(x)-\theta\|&\geq& \prod_{\theta} (\|a(x)\|-\|\theta\|)\\
&>& \prod_{\theta} (H_f+2-(H_f+1))=1,
\end{eqnarray*}
which shows that $\deg f_i(x,a(x))>0$ for each $i$. Consequently, it follows from the equality $f(x,a(x))=f_1(x,a(x))\cdots f_r(x,a(x))$ that there must be an irreducible factor of $f(x,a(x))$  dividing $f_i(x,a(x))$ for each $i$. This proves that $r\leq \nu_a$.
\end{proof}
To prove Theorem \ref{th5}, we will use the degree valuation $v$ on $\K[x]$ by defining $v(a(x))=-\deg a$ for all $a\in \K[x]$.  The Newton polygon $NP(f)$ of a polynomial $f=a_0(x)+a_1(x)y+\cdots+a_n(x)y^n\in \K[x,y]$ with $a_0a_n\neq 0$ is then defined as the lower convex hull of the points (called vertices of $NP(f)$) $(i,v(a_i))$, $0\leq i\leq n$ and $a_i\neq 0$. With these notions, we have the following bivariate version of Dumas Theorem \cite{Du} whose proof is parallel to that in the univariate case.
\begin{thmx}[Dumas]\label{thE}
Let $\K$ be a field. Let $g$, $h$ be nonconstant polynomials in $\K[x,y]$ such that $g(x,0)h(x,0)\neq 0$. Then the edges in the Newton polygon of $f(x,y)=g(x,y)h(x,y)$ with respect to the degree valuation $v$ may be formed by constructing a polygonal path composed by translates of all the edges that appear
in $NP(g)$ and $NP(h)$ with respect to $v$, using
exactly one translate for each edge, in such a way as to form a
polygonal path with increasing slopes.
\end{thmx}
Theorem \ref{thE} in particular tells us that if there is an edge $AB$ joining the vertices $A(a,b)$ and $B(a',b')$ of $NP(f)$ such that there is no vertex on $AB$ other than $A$ and $B$, then any factorization $f=gh$ of $f$ has a factor of degree at least $|a-a'|$. Observe that there is no vertex on the segment $AB$ other than $A$ and $B$ if and only if $\gcd(|a-a'|,|b-b'|)=1$. We will use these observations to prove Theorem \ref{th5} as follows.
\begin{proof}[\bf Proof of Theorem \ref{th5}]
Since $a_j\in \K^\times$, we have $v(a_j)=0$ and slope of the segment joining the points $A(\ell,v(a_\ell))$ and $B(j,v(a_j))$ in the plane is
\begin{eqnarray*}
\frac{v(a_j)-v(a_\ell)}{j-\ell}=\frac{\deg a_\ell}{j-\ell}=\max_{0\leq i\leq j-1}\Bigl\{\frac{\deg a_i}{j-i}\Bigr\}=\max_{0\leq i\leq j-1}\Bigl\{\frac{v(a_j)-v(a_i)}{j-i}\Bigr\},
\end{eqnarray*}
Consequently, the segment $AB$ is one of the edges of $NP(f)$. Since, $\gcd(j-\ell,v(a_\ell))=1$, the segment $AB$ does not contain any vertex of $NP(f)$ other than the end points $A$ and $B$. By Theorem \ref{thE}, any factorization $f=gh$ of $f$ has a factor of degree at least $j-\ell$. So, the degree of the other factor must be at most $n-(j-\ell)$, as desired.
\end{proof}
\section{Examples}\label{sec:3}
In this section, we will provide some explicit examples of polynomials whose factorization properties may be deduced using Theorems \ref{th1}-\ref{th5}.
\begin{example}
The polynomial $f=x^8 + 6 x^7 + 5$ takes the value $f(108)=109\times 179244006092537$, a product of two primes. By Theorem   \ref{th1}, the polynomial $f$ is a product of at most two irreducible polynomials in $\mathbb Z[x]$. Since $f(-1)=0$, it follows that $x+1$ is a factor of $f$, and so, $f$ is a product of exactly $2$ irreducible factors in  $\mathbb Z[x]$.
\end{example}
\begin{example}
For any odd prime $p$, the polynomial
\begin{eqnarray*}
f=p^2-2px+x^2-p^2x^{2n},
\end{eqnarray*}
satisfies the hypothesis of Theorem \ref{th2} with $H_f=1$, $d=1$, $m=p\geq 3=2+H_f$, since $f(p)=p^{2(n+1)}$ which is coprime to $f''(p)/2!$, since $f''(p)/2!\equiv 1 \mod p$. So, $f$ is a product of at most $\min\{2(n+1),2\}=2$ irreducible polynomials in $\mathbb{Z}[x]$. Since we have
\begin{eqnarray*}
p^2-2px+x^2-p^2x^{2n}=(p-x+px^n)(-p+x+px^n),
\end{eqnarray*}
it follows that $f$ is a product of exactly two irreducible polynomials. This further proves that each of the polynomials $p-x+px^n$ and $-p+x+px^n$ must be irreducible in $\mathbb{Z}[x]$.
\end{example}
\begin{example} The number 9841 has $3$-adic expansion
\begin{eqnarray*}
13\times 757&=&9841= 1+3+3^2+3^3+3^4+3^5+3^6+3^7+3^8,
\end{eqnarray*}
where $9841$ is a product of two prime numbers 13 and 757. By Theorem \ref{th4}, the polynomial
\begin{eqnarray*}
f=1+x+x^2+x^3+x^4+x^5+x^6+x^7+x^8
\end{eqnarray*}
is a product of at most $\nu_{9841}=2$ irreducible factors in $\mathbb{Z}[x]$. Further, since
\begin{eqnarray*}
  1+x+x^2+x^3+x^4+x^5+x^6+x^7+x^8 &=& (1+x+x^2)(1+x^3+x^6),
\end{eqnarray*}
the polynomial $f$ must be a product of exactly two irreducible factors, and so, each of the polynomials $1+x+x^2$ and $1+x^3+x^6$ must also be irreducible.
\end{example}
\begin{example}
For a prime number $p$, the bivariate polynomial
\begin{eqnarray*}
f(x,y)=1+pxy+py^n,~n\geq 2
\end{eqnarray*}
satisfies the hypothesis of Theorem \ref{th4} with $H_f=\rho^{\deg px}=\rho$, and $a(x)=x^2$, then on taking $\rho=2$, we have $\deg a(x)=2=\log (2+\rho)/\log \rho$. The polynomial $f(x,a(x))=f(x,x^2)=1+px^3+px^{2n}$ is irreducible in $\mathbb{Q}[x]$, since the reciprocal polynomial of $f(x,x^2)$ is Eisensteinian with respect to the prime $p$. So, $\nu_a=1$. Thus, the polynomial $1+pxy+py^n$ is irreducible in $\mathbb{Q}[x,y]$.
\end{example}
\begin{example}
Now consider the bivariate polynomial
\begin{eqnarray*}
h(x,y)=1+xy+y^n\in\mathbb{Q}[x,y],~n\geq 1.
\end{eqnarray*}
We will establish the irreducibility of $h$ via that of its reciprocal polynomial $\tilde{h}$ with respect to $y$ using Theorem \ref{thD} and Theorem  \ref{th4} as follows. In view of Theorem \ref{th4}, we take $\rho=2$ so that $H_{\tilde{h}}=2$ and $a(x)=2x^2$ so that $\deg a(x)=2=\log(2+H_{\tilde{h}})/\log 2$. By Theorem \ref{th4} $\tilde{h}(x,y)=1+xy^{n-1}+y^n$ is a product of at most $\nu_a$ factors in $\mathbb{Q}[x,y]$, where $\nu_a$ is the number of irreducible factors of $g(x)=\tilde{h}(x,a(x))=1+2^{n-1}x^{2n-1}+2^nx^{2n}$ in $\mathbb{Q}[x]$. Observe that $\tilde{g}(x)=2^n+2^{n-1}x+x^{2n}$, which satisfies the hypothesis of Theorem \ref{thD} for $j=2n$, $k=n-1$, $\ell=1$, $\gcd(j-\ell,k)=\gcd(2n-1,n-1)=1$. So, the polynomial $\tilde{g}$ has a factor of degree at most $2n-j+\ell=1$. Since $\tilde{g}$ is monic and $\pm 1$ do not satisfy it, there is no factor of degree 1 of $\tilde{g}$ in $\mathbb{Q}[x]$. Thus, the polynomial $\tilde{g}$ and hence $g$ must be irreducible in $\mathbb{Q}[x]$, which  established that $\nu_a=1$. Consequently the polynomial $\tilde{h}$ and hence $h$ is irreducible in $\mathbb{Q}[x,y]$.

Observe that the polynomial $h=1+xy+y^n$ satisfies the hypothesis of Theorem \ref{th5} for $\ell=1$ and $j=n$, and so, any factorization of $h$ in $\mathbb{Q}[x,y]$ has a factor of degree at most $n-j+\ell=1$. Suppose on the contrary that $h$ has a factor of degree 1 and $\deg h>1$ as a polynomial in $y$. Then $n>1$ and
\begin{eqnarray}\label{e1}
1+xy+y^n=(c_0+c_1y+\cdots+c_{n-1}y^{n-1})(b_0+b_1y),
\end{eqnarray}
for some $b_0,b_1,c_i\in \mathbb{Q}[x]$, $0\leq i\leq n-1$ with $c_{n-1}b_1\neq 0$. On comparing the coefficients of like powers of $y$ on both sides of \eqref{e1}, we find that $c_0=b_0=\pm 1$, $c_1=0$, and correspondingly $b_1=\pm x$. Consequently,   $1=b_1 c_{n-1}=\pm c_{n-1}x$, which is absurd. We conclude that $h$ has no factor of degree 1 in $\mathbb{Q}[x,y]$, and so, $h$ is irreducible.

We mention here that a direct proof of the irreducibility of $1+xy+y^n$ was given recently in  \cite[Example 5]{NBJRG2024} using Newton-polygons for bivariate polynomials.
\end{example}
%\subsection{Acknowledgements} The authors is indebted to the referees.....
\subsection*{Disclosure statement}
The author reports to have no competing interests to declare.

\end{document}